\documentclass{llncs}
\usepackage[utf8]{inputenc}
\usepackage{enumerate}
\usepackage{graphicx}
\usepackage{epstopdf}
\usepackage{amsfonts}
\usepackage{amsmath}
\usepackage{amssymb}
\usepackage{sectsty}
\usepackage{hyperref}
\usepackage{url}
\usepackage{xspace}
\usepackage{color}
\usepackage{parskip}

\def\dist{\textrm{dist}}
\newcommand{\On}{\mathrm{On}}

\begin{document}

\title{
Maximum Wiener index of unicyclic graphs with given bipartition
}

\author{Jan Bok \inst{1}, Nikola Jedli\v{c}kov\'{a} \inst{2} \and Jana Maxov\'{a} \inst{3} }
\institute{
Computer Science Institute, Faculty of Mathematics and Physics, Charles University, Malostransk\'{e} n\'{a}m\v{e}st\'{i} 25, 11800, Prague, Czech Republic. Email: \email{bok@iuuk.mff.cuni.cz}
\and
Department of Applied Mathematics, Faculty of Mathematics and Physics, Charles University, Malostransk\'{e} n\'{a}m\v{e}st\'{i} 25, 11800, Prague, Czech Republic. Email: \email{jedlickova@kam.mff.cuni.cz}
\and
Department of Mathematics, Faculty of Chemical Engineering, University of Chemistry and Technology,
Technick\'{a} 5, 166 28, Prague, Czech Republic.
Email: \email{maxovaj@vscht.cz}
}

\maketitle

\begin{abstract}
The \emph{Wiener index} is a widely studied topological index of graphs. One of the
main problems in the area is to determine which graphs of given properties
attain the extremal values of Wiener index.

In this paper we resolve an open problem posed by Du in \emph{[Wiener indices of trees and monocyclic graphs with given bipartition.
International Journal of Quantum Chemistry, 112:1598–1605, 2012]}. To this end we determine the unicyclic bipartite graphs
with given size of parts having the maximum Wiener index. This completes the previous 
research in which the minimum case was solved.

\noindent{\bf 2010 Mathematics Subject Classification:}\enspace 05D99, 94C15
\vspace{0.5em}\\
\noindent{\bf Keywords:}\enspace Wiener index, topological indices, chemical graph theory, extremal graph theory, unicyclic graphs
\end{abstract}

\section{Introduction}

The \emph{Wiener index} (also \emph{Wiener number}) was introduced by Harry Wiener \cite{wiener} as path number in 1947 to study the boiling points of paraffins. The Wiener index $W(G)$ of
a connected graph $G=(V,E)$  is defined as
$$W(G) := \sum_{\{u,v\} \subseteq V}\dist_G(u,v).$$
In other words, it is the sum of the lengths of the shortest paths
between all unordered pairs of vertices in the graph.

The Wiener index can be defined also in terms of \emph{transmission}.

\begin{definition}
Let $G$ be a graph. The \emph{transmission} of $v \in V(G)$ in $G$ (we denote it by $t_G(v)$) is the sum
of distances between $v$ and all other vertices of $G$, i.e.
$$t_G(v) =  \sum_{u \in V(G)} \dist_G(u,v).$$
\end{definition}

We can now write $W(G)$ alternatively as $W(G) := \frac12 \sum_{v \in
V(G)}t_G(v)$.

Wiener index is one of the most studied topological indices in mathematical
chemistry and it remains a very active research topic.
As a purely mathematical concept, Wiener index was studied under different
names --- gross status \cite{harary1959status}, the transmission \cite{vsoltes1991transmission} and the distance of graphs \cite{entringer1976distance}. Practical applications of the invariant are wide and apart from chemistry, there are applications in biology, communication theory, facility location,
cryptology, just to name a few. We refer the reader to the numerous surveys, e.g. \cite{dobrynin2001wiener,knor2014wiener,Knor2016,Xu2014461,bonchev2002wiener}.

Extremal graph theory is a branch of mathematics concerned with finding the extremal
graphs satisfying certain given set of properties (for more, see e.g.
\cite{bollobas2004extremal}). Many problems regarding Wiener index fall into
this category of problems. By our opinion, this is nicely illustrated by the
survey of Furtula et al.\ \cite{Xu2014461}. Also, as is seen in
\cite{Knor2016}, there are many long-standing open problems in the area.

We recall that a graph is \emph{unicyclic} (or monocyclic) if it is connected
and contains exactly one cycle. There is a rich literature on extremal problems
regarding Wiener index on unicyclic graphs, e.g.\ \cite{dong2012maximum,hou2012maximum,liu2011wiener}. Section 4 of the aforementioned
paper of Furtula et al.\ \cite{Xu2014461} is devoted to unicyclic graphs.

Furthermore, we say that a graph on $p+q$ vertices has a \emph{$(p,q)$-partition} if it is bipartite
and the parts of the bipartition are of size $p$ and $q$.

Regarding the bipartite unicyclic graphs with given bipartition $p,q$ and $p \le q$, Du
\cite{Du2012} showed that the minimum Wiener index is attained by the graph
which is obtained by connecting $p-2$ vertices to one vertex of a $4$-cycle,
and connecting $q-2$ vertices to its neighbor on the $4$-cycle. Moreover, if
$p = q = 3$, then $C_6$ is also an extremal graph.

But the maximum case was left open.

\begin{problem} \cite{Du2012}
Find the graph(s) with the maximum Wiener index among unicyclic graphs on $n$
vertices with bipartition sizes $p$ and $q$, where $n=p+q$.
\end{problem}

Our main result is the final resolution of this problem. We first need to define \emph{onion graphs}.

\begin{definition}
The \emph{onion graph} $\On(k,l,m)$ is formed by a cycle of size four with
antipodal vertices $u,v$, a set of $k$ pendant edges attached to the vertex $v$,
path $P_l$ with one of the endpoints identified with the vertex $u$ and
with $m$ pendant edges attached to the second endpoint of $P_l$. (For $l=1$,
the endpoints coincide.)
\end{definition}

\begin{figure}[h]
    \centering
    \includegraphics[angle = 270, scale = 0.8]{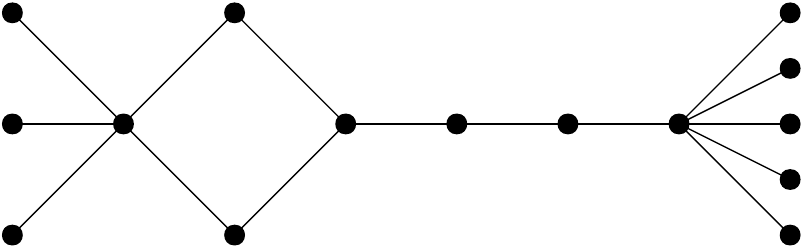}
    \caption{An example of an onion graph -- $\On(3,4,5)$.}
    \label{fig:priklad}
\end{figure}

With this definition in hand we can now state our main theorem.

\begin{theorem} \label{thm:main}
    The maximum Wiener index for $n$-vertex bipartite unicyclic graphs
    with given size of partitions $p,q$ ($1 < p \le q$) is attained by
    \emph{exactly one} graph:
    $$\On(\lfloor (q-p)/2 \rfloor, 2p-3, \lceil (q-p)/2 \rceil).$$
    Its Wiener index is equal to
\begin{align*}
&(2p-5)\cdot\lceil(q-p)/2\rceil\lfloor(q-p)/2\rfloor + (p-7)\lceil(q-p)/2\rceil + (13-7p)\lfloor(q-p)/2\rfloor + \\
&2p^2q + (q-p)^2 + 2p^3 - 37p + 66.
\end{align*}
\end{theorem}

\section{Lemmata}

We will prove the main theorem with the help of the following lemmata. Each
one yields some property of the extremal graphs. Finally, we will collect these
lemmata and conclude that there is a unique graph with the maximum Wiener index for given parameters. 

Throughout this paper, all graphs we consider are simple and connected.
For a graph $G=(V,E)$, we write $|G|$ for $|V(G)|$. We also shorten $\{u,v\}$
for an edge to $uv$. The operation of vertex and edge insertion is written as $G \cup v$ and $G \cup uv$, respectively.
For vertex and edge deletion, we write analogously $G - v$ and $G - uv$, respectively.

We denote by $E_{p,q}$ (for $1 < p \leq q$ and $p,q \in \mathbb N$) the set of
extremal graphs having the maximum possible Wiener index among all bipartite unicyclic
graphs on $p+q$ vertices with $(p,q)$-partition. With this
notation in hand, we can rephrase our aim as a search for characterizing the
set $E_{p,q}$ for every possible parameters $p$ and $q$.

We often use $P$ and $Q$ for the set of vertices in the part of size $p$ and $q$,
respectively.

Consider some unicyclic graph $U$ having a unique cycle $C$. We say that $T$ is
a \emph{rooted tree} if it is a tree component of graph $U'$ rooted in some
vertex $v \in V(C)$, where $V(U') = V(U)$ and $E(U')=E(U)-E(C)$.

\begin{lemma}
Every extremal graph $G \in E_{p,q}$ contains $4$-cycle.
\end{lemma}

\begin{proof}
Suppose for a contradiction that there exists a $G \in E_{p,q}$ without a
$4$-cycle. The graph $G$ clearly contains a $k$-cycle (with $k$ even and bigger than
$4$) as a subgraph. Denote consecutively $v_1,\ldots,v_k$ the vertices of the
cycle. Furthermore, denote by $T_i$ the tree rooted in $v_i$. 
Note that for all $i=1,\ldots,k$,
the number of vertices in $T_i$ is at least one as $v_i \in T_i$. Without loss of
generality we may suppose that $|T_k| \geq |T_i|$ for every $i=1,\ldots, k-1$.

We define a new graph $G'$ by setting $V(G'):=V(G)$ and $E(G'):=(E(G) \cup v_1
v_4) - v_1 v_k$. The subgraph of $G'$ induced by the vertices $v_1, \dots v_k$
is formed by a $4$-cycle and a path of length $k-4$ attached to $v_4$. Since
$v_k$ and $v_4$ belong to the same part of $G$, both $G$ and $G'$ have the
same $(p,q)$-partition.

For every $a \in T_1 \cup T_2$ and every $b \in T_k$, it holds that $$\dist_{G'}(a,b) =
\dist_{G}(a,b)+k-4.$$ Note that for $k>6$, there can be some other pairs of vertices such
that $\dist_{G'}(a,b) > \dist_{G}(a,b)$ but to prove this lemma
it is sufficient to consider just the aforementioned pairs.

Let $x$ and $y$ be arbitrary vertices in $G$ such that
$\dist_{G'}(x,y) < \dist_{G}(x,y)$. Then by the construction of $G'$, one of the vertices $x$
and $y$ belongs to $T_1$. Without loss of generality we may suppose that $x
\in T_1$.

Let us define $I: = \{4,\dots ,\frac k2 +1\}$. It is clear that for any
$x \in T_1$ and any $y \in V(G')$, it holds that $\dist_{G'}(x,y) < \dist_{G}(x,y)$ if
and only if $y \in T_i$ with $i \in I$. Moreover, if $x \in T_1$ and  $y \in
T_i, i \in I$, then  $$\dist_{G'}(x,y) = \dist_{G}(x,y) -2.$$

Our goal is to prove that $W(G')>W(G)$. From the previous observations we get
$$W(G')-W(G) \geq (k-4) \,|T_1| \cdot |T_k|  + (k-4) \, |T_2|\cdot |T_k| -2
\sum_{i \in I} |T_1|\cdot |T_i|.$$  From the assumption $|T_k| \geq |T_i|$ for
every $i=1, \dots , k$ and  $|I|= \frac k2 -2 = \frac {k-4}2$, it follows that
\begin{align*}
W(G')-W(G)&\geq (k-4)\,|T_1|\cdot |T_k|  + (k-4)\,|T_2|\cdot |T_k| -2  \left( \frac{k-4}2 \right)  |T_1|\cdot |T_k| \\
&=(k-4) \, |T_2|\cdot |T_k|.
\end{align*}

As $|T_i| \geq 1$ for all $i$ and $k \geq 6$, we get $W(G')-W(G) >0$, a
contradiction with the extremality of $G$.
\end{proof}

In the next proofs we use the following lemmata from the paper of Du \cite{Du2012} and from the paper of Polansky \cite{polansky1986}, respectively.

\begin{lemma} \cite{Du2012} \label{lem:du}
Let $G,H$ be two nontrivial connected graphs with $u,v \in V(G)$ and $w \in V(H)$. Let $GuH$ ($GvH$, respectively) be the graph obtained from $G$ and $H$ by identifying $u$ ($v$, respectively) with $w$. If $t_G(u) < t_G(v)$, then $W(GuH) < W(GvH)$.
\end{lemma}

\begin{lemma} \cite{polansky1986} \label{lem:polansky}
Let $G_u$ and $G_v$ be two graphs with $n_u$ and $n_v$ vertices, respectively, and let $u \in V(G_u)$ and $v \in V(G_v)$. If $G$ arises from $G_u$ and $G_v$ by identifying $u$ and $v$, then
$$W(G) = W(G_u) + W(G_v) + (n_u - 1)t_{G_v}(v) + (n_v - 1)t_{G_u}(u). $$
\end{lemma}

\begin{lemma} \label{lem:jeden_nenaveseny}
Every extremal graph $G \in E_{p,q}$ has at least one vertex of degree two on its cycle.
\end{lemma}

\begin{proof}
Assume that $G \in E_{p,q}$ and
denote by $v_1,\ldots, v_4$ vertices on its cycle. We assume for a
contradiction that there is a rooted tree $T_i$ for each $v_i, i \in
\{1, \ldots,4\}$ such that $|T_i| \ge 2$.

Let us define graphs $G':= G \setminus (T_1-v_1)$ and $H := T_1$.
Without loss of generality we may assume that $v_1$ is
in part $P$ and $|T_2| \leq |T_4|$.

We first observe that the transmission of $v_1$ in $G'$ can be written as
$$ t_{G'}(v_1) = t_{T_2}(v_2) + t_{T_3}(v_3) + t_{T_4}(v_4) + |T_2| + 2|T_3| + |T_4|.$$

\begin{figure}[h]
    \centering
    \includegraphics[scale=0.7]{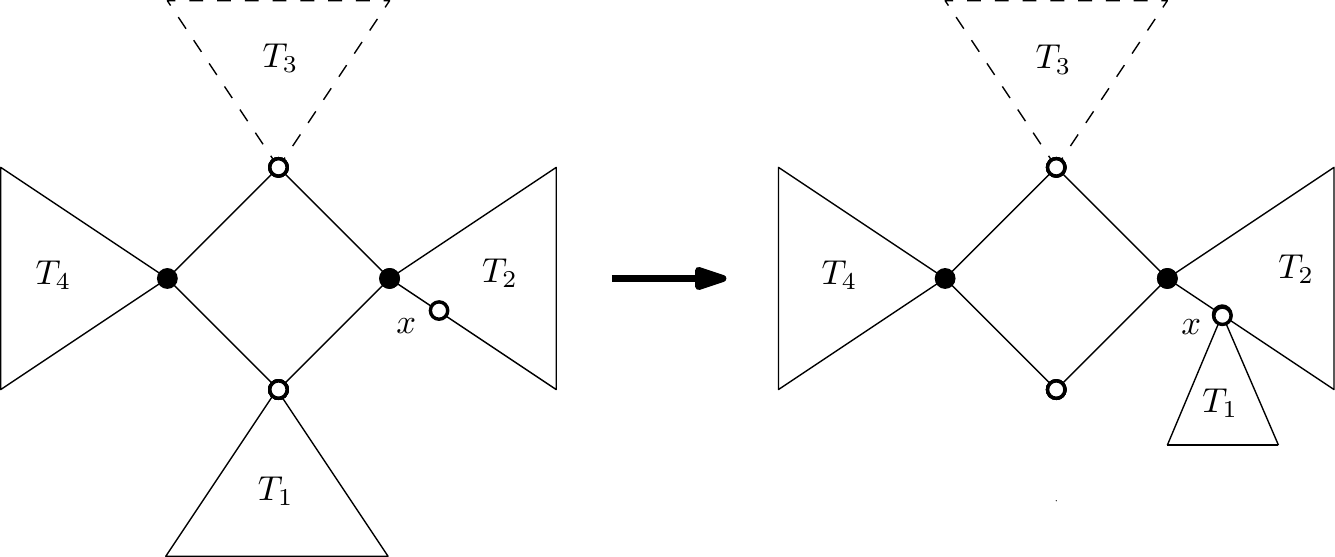}
    \caption{An illustration of situation in Lemma~\ref{lem:jeden_nenaveseny}.}
    \label{fig:antipodal}
\end{figure}

Let $x \in V(T_2)$ be any vertex of $T_2$ such that $v_2x \in E(G)$. Such
vertex must exist because $|T_2| \ge 2$.  Our aim is to prove that
$t_{G'}(v_1) < t_{G'}(x)$ and consequently apply Lemma~\ref{lem:du}.

Denote by $T_x$ the tree component of $G-v_2x$, see Figure~\ref{fig:antipodal}. Note that both $x$ and $v_1$ are in part $P$. The transmission of $x$ in $G'$ is equal to
$$t_{G'}(x) = t_{T_x}(x) +t_{T_2 \setminus T_x}(v_2) + |T_2 \setminus T_x |+ t_{T_3}(v_3) + t_{T_4}(v_4) +
 2|T_3| + 3|T_4| + 2.$$

It is easy to see that $t_{T_2}(v_2) = t_{T_x}(x) +t_{T_2 \setminus T_x}(v_2) + |T_x|$
and hence it follows
$$t_{G'}(x) = t_{T_2}(v_2) + t_{T_3}(v_3) + t_{T_4}(v_4) + |T_2| + 2|T_3| + 3|T_4| - 2|T_x| + 2.$$

By observing that $T_x \subset T_2$ and combining this with the assumption
$|T_2| \leq |T_4|$ we get
$$t_{G'}(x) - t_{G'}(v_1) = 2|T_4| - 2|T_x| + 2 > 0,$$
from which we conclude that $t_{G'}(v_1) < t_{G'}(x)$.

Let $G'v_1H$ be a graph obtained from $G'$ and $H$ by identifying $v_1$ in
$G'$ and $v_1$ in $H$ and let $G'xH$ be a graph obtained from $G'$ and $H$ by
identifying $x$ in $G'$ with $v_1$ in $H$. It is clear that $G'v_1H \cong G$.
By Lemma~\ref{lem:du} we get $W(G) = W(G'v_1H ) < W(G'xH)$, a contradiction with
$G$ being extremal.
\end{proof}

\begin{lemma}\label{lem:antipodal}
Every extremal graph $G \in E_{p,q}$ has two antipodal vertices of degree 2 on its cycle.
\end{lemma}

\begin{proof}
We assume for a contradiction that $G \in E_{p,q}$ and that it does
not have two antipodal vertices on its cycle with degree 2. We distinguish two
cases.

\emph{Case 1.}
Graph $G$ has only one vertex of degree two on its cycle. Denote the vertices
of cycle by $v_1, \ldots, v_4$ consecutively so that the vertex of degree
two is denoted by $v_3$.  We can proceed in the same way as in the proof of
the previous lemma to get a contradiction.

\emph{Case 2.}
Graph $G$ has exactly two vertices of degree two on its cycle and they are
adjacent. We denote the vertices on the cycle $v_1,\ldots v_4$ consecutively
such that $v_3$ and $v_4$ have degree two. There is a tree $T_1$
rooted in $v_1$ such that $|T_1|\ge2$ and a tree $T_2$ rooted in $v_2$ such that $|T_2|\ge2$.

Again, we distinguish two cases.

\emph{Case 2a.}
We assume that at least one of $v_1$ and $v_2$ has degree at least 4.
Without loss of generality suppose that $v_2$ is the vertex. Denote by $x$ any
of its neighbors in $T_2$ and by $T_x$ the tree component of $G-v_2x$.

We define graphs $G' :=  G \setminus T_x $ and $H := T_x \cup v_2x$. Note
that $v_2$ is included in both of $G'$ and $H$ but the edge $v_2x$ is only in
$E(H)$. See Figure \ref{fig:lemma_5_case2a}. We will proceed in a similar way as in the proof of the previous
lemma. We define $T_2':= T_2 \setminus T_x$.  It holds that
\begin{align*}
 t_{G'}(v_2) &= t_{T_2'}(v_2) + t_{T_1}(v_1) + |T_1| + 3, \textrm{ and} \\
 t_{G'}(v_4) &= t_{T_2'}(v_2)+ 2|T_2'| + t_{T_1}(v_1) + |T_1| + 1.
\end{align*}

\begin{figure}[h]
    \centering
    \includegraphics[scale = 0.75]{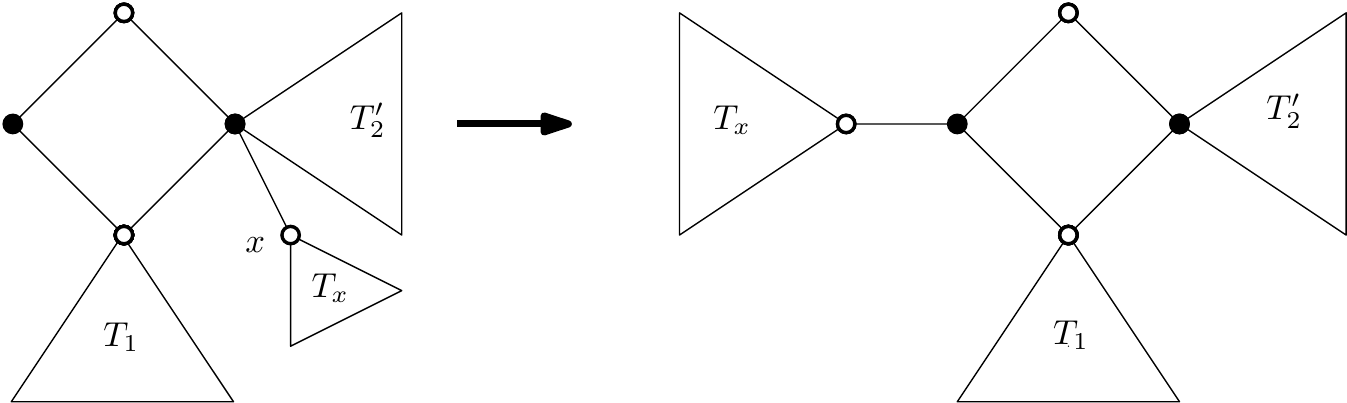}
    \caption{An illustration of construction in Case 2a of Lemma~\ref{lem:antipodal}.}
    \label{fig:lemma_5_case2a}
\end{figure}

We assume that $v_2$ has degree at least 4 and thus $|T_2'| \geq 2$. It follows
$$ t_{G'}(v_4) - t_{G'}(v_2) = 2|T_2'| - 2 > 0,$$
thus $t_{G'}(v_2) < t_{G'}(v_4)$. 

Let $G'v_2H$ be a graph obtained from $G'$ and $H$ by identifying $v_2 \in
V(G')$ and $v_2 \in V(H)$ and let $G'v_4H$ be a graph obtained from $G'$ and
$H$ by identifying $v_4 \in V(G')$ with $v_2 \in V(H)$. Observe that $G'v_2H
\cong G$. It follows from Lemma~\ref{lem:du} that $W(G'v_2H ) < W(G'v_4H)$,
a contradiction because $G$ was extremal.

\emph{Case 2b. }

We assume that both $v_1$ and $v_2$ have degree 3. Without loss of generality
we assume $|T_1| \leq |T_2|$. Let us denote by $x$ the vertex adjacent to
$v_1$ that is not on the cycle and denote by $T_x$ the tree component of
$G-v_1x $.  Similarly, let us denote by $y$ the vertex adjacent to $v_2$ that
is not on the cycle and by $T_y$ the tree component of $G-v_2y$. As $|T_1|
\leq |T_2|$ also $|T_x| \leq |T_y|$. Let $G''$ be the graph obtained from $G$
by deleting edge $v_2y$, adding edge $xy$ and edges from $v_2$ to all
neighbors of $x$ in $T_x$, and, finally, by deleting edges between $x$ and its
neighbors in $T_x$. Note that $G''$ has the same bipartition as $G$. See
Figure \ref{fig:lemma_5_case2b}. We calculate the Wiener indices of $G$ and
$G''$

\begin{figure}[h]
    \centering
    \includegraphics[scale = 0.75]{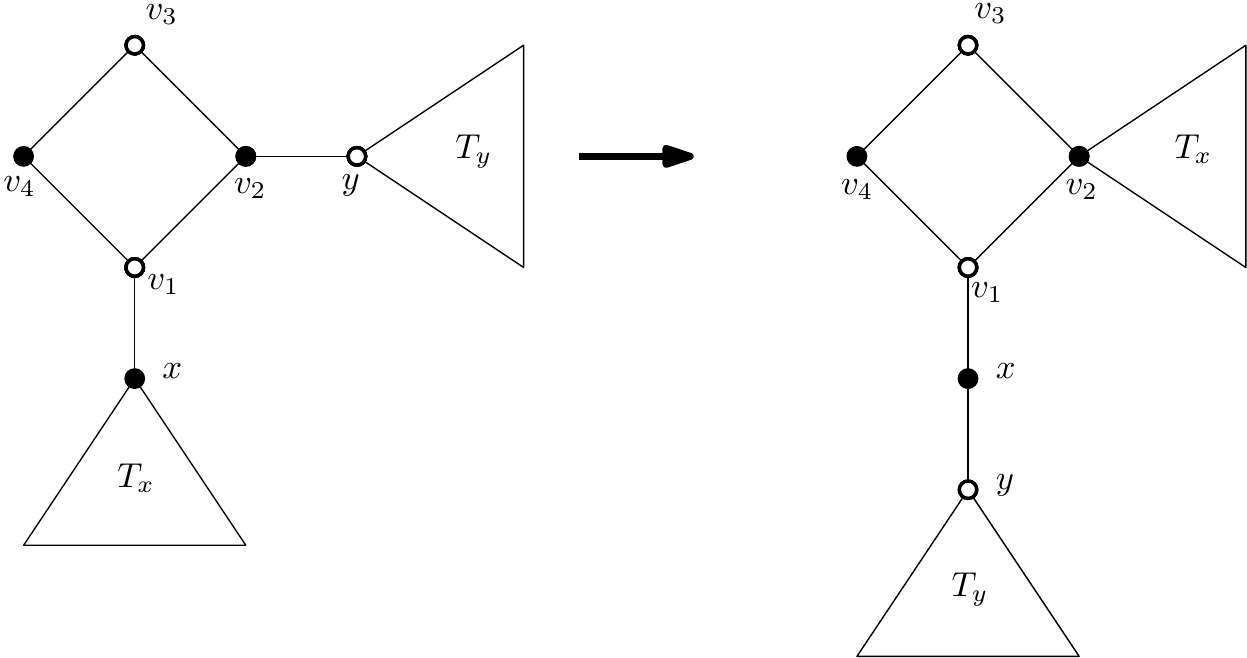}
    \caption{An illustration of construction in Case 2b of Lemma~\ref{lem:antipodal}.}
    \label{fig:lemma_5_case2b}
\end{figure}

\begin{align*}
W(G) &= \sum_{u \in T_x, v \in T_y} \dist_G(u,v) +4t_{T_x}(x) + |T_x|(1+4+3) \\
&\quad + 4t_{T_y}(y) + |T_y|(1+4+3)+8,\\
W(G'')&= \sum_{u \in T_x, v \in T_y} \dist_{G''}(u,v)+ 4t_{T_x}(x) + |T_x|(2+4)\\ 
&\quad +4t_{T_y}(y) + |T_y|(1+2+3+4)+10.
\end{align*}

Note that for any $u \in T_x$ and any $v \in T_y$, the distances in $G$ and
$G''$ are the same, i.e.\ $\dist_G(u,v)=\dist_{G''}(u,v)$. We
conclude that $$W(G'')-W(G)=2|T_y|-2|T_x|+2>0,$$ a contradiction with the
extremality of $G$.
\end{proof}

So far, it follows from the previous lemmata that for arbitrary $p,q\ge2$ every graph in $E_{p,q}$
contains a cycle $C_4$ with two antipodal vertices on $C_4$ of degree $2$.

For the rest of our paper we need a definition of \emph{broom} graph.

\begin{definition}
For any $a \geq 1$ and $b \geq 0$, we say that a tree $T$ is a \emph{broom} if $T$ arises from a path $P=x_1, \dots x_a$  on $a$ vertices by adding $b$ pendant vertices to $x_a$. We say that $x_1$ is the \emph{root} of the broom.
\end{definition}

\begin{lemma}
If $G \in E_{p,q}$ is a graph with a tree $T_1$ attached to $v_1$ and a $T_2$ attached to $v_3$, then both trees $T_1$ and $T_2$ are isomorphic to brooms with roots in $v_1$, $v_3$, respectively.
\end{lemma}

\begin{proof}
 We
denote the vertices of $C_4$ subgraph of $G$ by $v_1, \ldots, v_4$ consecutively such that
$v_2$ and $v_4$ have degree two.

Let $P= u_1, \dots, u_t$ be a longest path in $G$. We say that a vertex $v \in V(G)$ is \emph{bad} if 
\begin{itemize}
\item $\mathrm{deg}_G(v)>2$ for $v \in V(G) \setminus \{u_2,u_{t-1},v_1,v_3\}$, or 
\item $\mathrm{deg}_G(v)>3$ for $v \in  \{v_1,v_3\} \setminus \{u_2,u_{t-1}\}$.
\end{itemize}
Suppose for a contradiction that $G \in E_{p,q}$ and trees attached to $C_4$
in $G$ are not isomorphic to brooms. Thus there exits a bad vertex in path
$P$, different from $v_2$ and $v_4$ since these are of degree two. Without
loss of generality we may suppose that the bad vertex is in $V(T_1)$. If there
is more than one bad vertex in $V(T_1)$, we will choose the one such that its
distance $u_2$ is minimal possible and we denote such bad vertex as $u_i$.
Therefore, the subgraph of $G$ induced by $\{u_2, \dots, u_i\}$ is a path and
$i \geq 3$.

We set $A := N(u_i) \setminus \{u_{i-1}, u_{i+1}, v_2, v_4\}$.  Let $H_1$ be
the component of $G \setminus A$ containing the vertex $u_i$ and  let $H_2$ be
the component of $G \setminus \{u_{i-1}, u_{i+1}, v_2, v_4\}$ containing the
vertex $u_i$. Note that $G$ arises from $H_1$ and $H_2$ by identifying $u_i
\in V(H_1)$ with $u_i \in V(H_2)$, or shortly $G=H_1 u_i H_2$. Let $K$ be the
component of $G \setminus u_i$ that contains $u_{i+1}$. Let $k := |V(K)|$ and
$d = \mathrm{deg}_{H_1}(u_2)$. We distinguish three cases.

\emph{Case 1. } $k < d$

We set $G_1:= H_1 u_{i+2} H_2$. In order to prove that $W(G_1) >W(G)$ we
compute the difference $t_{H_1}(u_{i+2})-t_{H_1}(u_{i})$ and apply
Lemma~\ref{lem:du}.

Note that there are at least $i+d-2 >d$ vertices $v$ in $V(H_1)$ such that
their distance to $u_{i+2}$ is bigger than their distance to $u_i$, namely
$\dist_{H_1}(v,u_{i+2})=\dist_{H_1}(v,u_i)+2 $. Further note
that
$\dist_{H_1}(u_{i+1},u_{i})=\dist_{H_1}(u_{i+1},u_{i+2})=1$
and that for at most $k-1$ vertices $v$ in $V(H_1)$  ($v \in V(K) \setminus
u_{i+1}$), the distance $\dist_{H_1}(v,u_{i+2}) \geq
\dist_{H_1}(v,u_i)-2$.

Thus we obtain that
$$t_{H_1}(u_{i+2})-t_{H_1}(u_{i}) \geq 2d -2(k-1)=2(d-k)+2>0,$$
from which we conclude $t_{H_1}(u_{i+2})>t_{H_1}(u_{i})$.

Observe that $u_i$ and $u_{i+2}$ are in the same part of $G$. Hence, $G_1$ also
has a $(p,q)$-bipartition. Again, we conclude that $W(G_1) >W(G)$, a
contradiction.

\emph{Case 2. } $k \geq d$ and $\dist_G(u_2,u_i)$ is even.

In this case we set $G_2:= H_1 u_{2} H_2$. We aim to prove that $W(G_2) >W(G)$
by computing the difference $t_{H_1}(u_{2})-t_{H_1}(u_{i})$ and using
Lemma~\ref{lem:du}.

Note that there are exactly $d-1$ vertices $v$ (the neighbors of $u_2$) in
$V(H_1) \setminus \{ u_2, \dots u_i \}$ such that
$\dist_{H_1}(v,u_i)=\dist_{H_1}(v,u_2)+(i-2)$. Further note that for any $v
\in V(K)$, the distance $\dist_{H_1}(v,u_i)=\dist_{H_1}(v,u_2)-(i-2)$. Thus
$$t_{H_1}(u_{2})-t_{H_1}(u_{i})= \sum_{v \in V(H_1) \setminus \{u_2, \dots,
u_i\}} \left(\dist_{H_1} (v,u_2) - \dist_{H_1} (v,u_i)\right) = $$ $$ =
k(i-2)-(d-1)(i-2) = (i-2)(k-d+1) >0$$  as $i \geq 3$ and $k\geq d$. It follows
that  $t_{H_1}(u_{2})>t_{H_1}(u_{i})$. Note that again $u_i$ and $u_{2}$ are
in the same partition of $G$ and thus $G_2$ also has a $(p,q)$-bipartition. By
Lemma~\ref{lem:du} we obtain that $W(G_2) >W(G)$, a contradiction.

\emph{Case 3. }  $k \geq d$ and $\dist_G(u_2,u_i)$ is odd.

We set $G_3:= H_1 u_{1} H_2$. As $u_i$ and $u_{1}$ are in the same partition
of $G$ and the graph $G_3$ also has a $(p,q)$-bipartition. Observe that
$t_{H_1} (u_1)= t_{H_1} (u_2) + |V(H_1)|-2$. This implies that
$t_{H_1}(u_{1})>t_{H_1}(u_{2})>t_{H_1}(u_{i})$ and again by
Lemma~\ref{lem:du} we obtain that $W(G_3) >W(G)$, a contradiction.
\end{proof}

\begin{lemma}
Let $G \in E_{p,q}$. If $|P|<|Q|$. Then all the pendant vertices belong to part $Q$.
\end{lemma}
\begin{proof}
Suppose for a contradiction that there are two pendant vertices
$x$ and $y$ such that $x \in P$ and $y \in Q$ (or vice versa). Clearly,
$x$ and $y$ belong to distinct brooms of $G$. Let $u$ be the only
neighbor of $x$ and $v$ be the only neighbor of $y$ in $G$. Without
loss of generality we may assume that $t_G(u) \le t_G(v)$.
Clearly $t_G(y) = t_G(v) + n - 2$. We define a graph $G' = G - x$. It holds that
\begin{align*}
 t_{G'}(u) &= t_G(u) - 1, \\
 t_{G'}(y) &\geq t_G(y) - n + 2.
\end{align*}
Therefore, we have
$$ t_{G'}(u) < t_{G}(u) \leq  t_{G}(v) = t_G(y) - n + 2 \leq t_{G'}(y).$$
Clearly, $G \cong G'uK_2$. By Lemma~\ref{lem:du} we have $W(G) = W(G'uK_2) < W(G'vK_2)$, a contradiction. 
\end{proof}

We define the \emph{height} of a rooted tree as the number of edges in the longest path between the root
and a leaf of the tree.

\begin{lemma} \label{lem:vyska_max_2}
Let $G \in E_{p,q}$ be an extremal graph. At least one of the trees (brooms) attached to $C_4$ in $G$ has the height at most two.
\end{lemma}
\begin{proof}
From the previous lemmata it follows that $G$ consists from $C_4$ with two brooms
attached to antipodal vertices of the cycle (we denote these vertices $v_0$
and $u_0$). We denote by $w$ one of the vertices incident with $v_0$ and $u_0$
on the cycle and by $w'$ the other one. Let $v_0, \ldots, v_i$ be the vertices
of the path in the broom attached to $v_0$ and $a$ be the number of
pendant vertices attached to $v_i$. Similarly, we denote by $u_0, \ldots, u_j$ the
vertices on the path in the broom attached to $u_0$ and by $b$ the
number of pendant vertices adjacent to $u_j$. See Figure~\ref{fig:lemma_vyska} for
an illustration.

\begin{figure}[h]
    \centering
    \includegraphics[scale = 0.75]{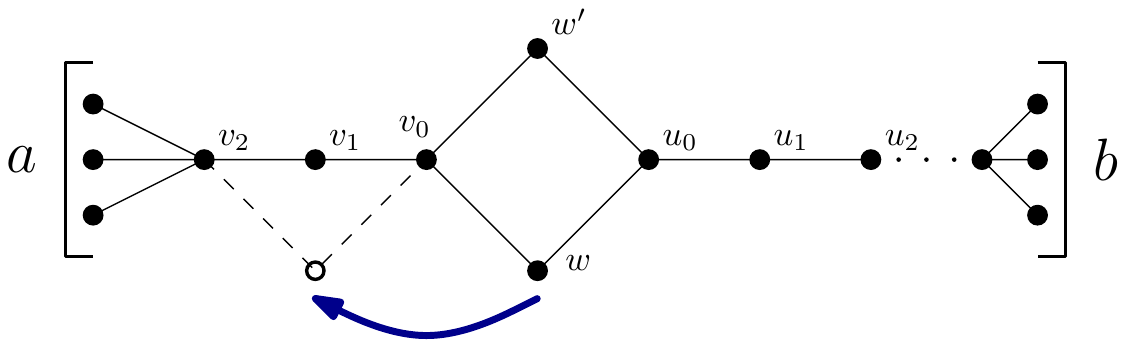}
    \caption{A possible scenario in Lemma~\ref{lem:vyska_max_2}.}
    \label{fig:lemma_vyska}
\end{figure}

We assume for a contradiction that $i,j \geq 2$.
Without loss of generality, $a \leq b$. We define the graph $G':= G \setminus w$. We divide the rest of the
proof into two cases, based on the parity of $i$.

\emph{Case 1. }
We assume that $i$ is even. Let $G_1$ be a graph obtained from $G'$ by joining a vertex $w$ to $v_i$ and $v_{i-2}$. Note that $G_1$ has the same partition as $G$ and $G_1$ has one of the brooms of height one.
Observe that for Wiener indices of $G$ and $G_1$ it holds that
\begin{align*}
W(G) &= W(G') + t_G(w), \\
W(G_1) &= W(G') + t_{G_1}(w),
\end{align*}
because the distances in $G'$ do not change when we add the vertex $w$. Now we calculate the transmission of $w$ in $G$ and $G_1$.
\begin{align*}
 t_{G}(w) &= 2 + {i+2 \choose 2} + a(i+2) + {j+2 \choose 2} + b(j+2), \\
 t_{G_1}(w) &= 3 + 2a + {i+j+2 \choose 2} +  b(i+j+2).
\end{align*}
Since we assume $a \leq b$ and $i,j \geq 2$, we get
$$ t_{G_1}(w) -  t_{G}(w) \geq 1 + i(b-a) > 0$$
and hence $W(G_1) > W(G)$, a contradiction.

\emph{Case 2. }
We assume that $i$ is odd. We will proceed in a similar way as in the first case. Let $G_2$ be a graph obtained from $G'$ by joining a vertex $w$ to $v_{i-1}$ and to vertex $v_{i-3}$.
Note that in this case we may assume that $i \ge 3$ as the situation for $i = 1$ is clear. Thus, $v_{i-3}$ exists.

$W(G)$ and $t_G(w)$ are the same as in the previous case and by a similar argument we get
\begin{align*}
W(G_2) &= W(G') + t_{G_2}(w) \\
t_{G_2}(w) &= 5 + 3a + {i+j+1 \choose 2} + b(i+j+1).
\end{align*}
By an easy calculation we get $t_{G_2}(w) - t_{G}(w) > 0$ which implies $W(G_2) > W(G)$, a contradiction.
\end{proof}

\begin{lemma} \label{lem:onions}
Let $G$ be an extremal graph. At least one of the trees (brooms) attached to $C_4$ in $G$ has the height at most one. In other words, $G$ is an onion graph.
\end{lemma}
\begin{proof}
We denote two antipodal vertices of degree bigger than two on $C_4$ by $u$ and $v$. We also denote the broom attached to $u$ by $B_u$ and the broom attached to $v$ by $B_v$.
It follows from Lemma~\ref{lem:vyska_max_2} that at least one broom has the height at most 2, without loss of generality it is $B_u$. If the height of $B_u$ is equal to one we are done. We assume that height of $B_u$ is two. Let $a$ be the number of pendant vertices in $B_u$. We distinguish two cases.

\emph{Case 1. } We assume that $a=1$. We define two graphs $G_1:= G \setminus (B_v-v)$ and
$H_1:= B_v$.  We denote by $y$ the only vertex of degree one in $G_1$.
Note that $G$ can be obtained from $G_1$ and $H_1$ by identifying
$v \in V(G_1)$ with $v \in V(H_1)$. Let $G'$ be a graph obtained from $G_1$ and $H_1$ by identifying
$y \in V(G_1)$ with $v \in V(H_1)$. See Figure~\ref{fig:lemma_onions_2} for an illustration of this transformation. Observe that $t_{G_1}(v) = 11$ and $t_{G_1}(y) = 13$. By Lemma~\ref{lem:du} we have $W(G') > W(G)$, a contradiction.

\begin{figure}[h]
    \centering
    \includegraphics[scale = 0.75]{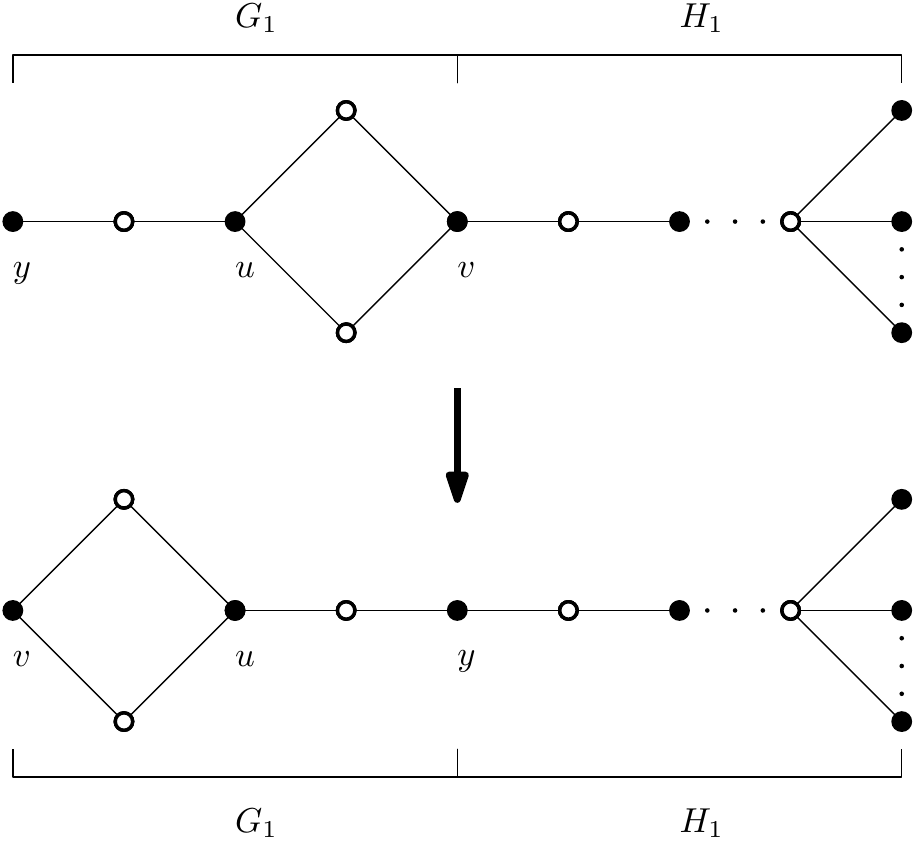}
    \caption{The transformation in the first case of Lemma~\ref{lem:onions}.}
    \label{fig:lemma_onions_2}
\end{figure}

\emph{Case 2. }
We assume that $a \geq 2$. Let $G_1$ and $H_1$ be defined as in the previous
case. We create a graph $G_2$ in the following way. Denote by $s$ and $t$ the
two antipodal vertices of degree two in the cycle in $G_1$. Attach to the
vertex $s$ exactly $a-2$ pendant vertices and attach to $t$ a path on four vertices $p_1, p_2, p_3,p_4$ with
endpoints $p_1$ and $p_4$ by identifying $p_1$ with $t$. We denote the resulting graph
by $G_2$. Note that $G_2$ has the same number of edges, the same number of
vertices and the same partition as $G_1$.

\begin{figure}[h]
    \centering
    \includegraphics[scale = 0.75]{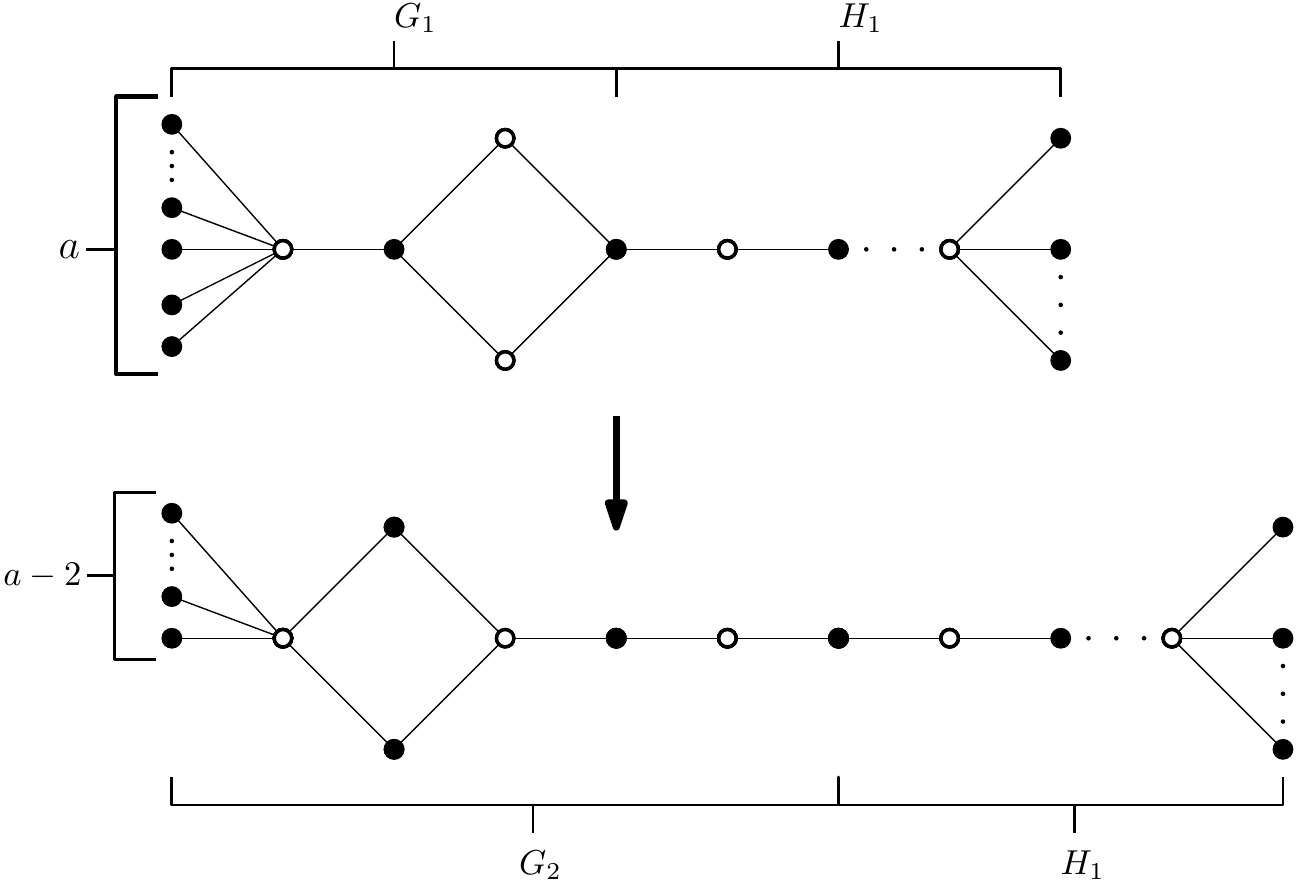}
    \caption{The transformation in the second case of Lemma~\ref{lem:onions}.}
    \label{fig:lemma_onions_3}
\end{figure}

Let $G'$ be a graph obtained from $G_2$ and $H_1$ by identifying $p_4 \in V(G_2)$ with $v \in V(H_1)$. See Figure~\ref{fig:lemma_onions_3} for an illustration.  By Lemma~\ref{lem:polansky} we get that for Wiener indices of $G$ and $G'$ it holds that
\begin{align*}
W(G) &= W(G_1) + W(H_1) + (|V(H_1)| - 1)t_{G_1}(v) + (|V(G_1)| - 1)t_{H_1}(v), \\
W(G') &= W(G_2) + W(H_1) + (|V(H_1)| - 1)t_{G_2}(v) + (|V(G_1)| - 1)t_{H_1}(v).
\end{align*}
We need to prove that
$$ W(G_2) - W(G_1) + (|V(H_1)|-1)\cdot (t_{G_2}(v) - t_{G_1}(v)) > 0.$$

First, we observe that for transmissions of $v$ it holds that
\begin{align*}
 t_{G_1}(v) &= 4a+7, \\
 t_{G_2}(v) &= 6a+7,
\end{align*}
which implies $(t_{G_2}(v) - t_{G_1}(v)) = 2a$.
Now we shall determine Wiener indices of $G_1$ and $G_2$.
Note that there are exactly $5+a$ edges in $G_1$, exactly $a + 4 + {a \choose 2}$
pairs of vertices in distance two, exactly $2a+1$ vertices in distance 3 in $G_1$, $a$ vertices in distance 4
and there is no pair of vertices in distance bigger than 4 in $G$.  Hence we have
\begin{align*}
W(G_1) &=  (5 + a) + 2\Big[(a+4) + {a \choose 2}\Big] + 3(2a+1) + 4a\\ &= 2{a \choose 2} + 13a + 16.
\end{align*}
Similarly, we get
\begin{align*}
W(G_2) &= 5 + a + 2(2a-2) + 2{a-2 \choose 2}\\ &+ 3(a+2) + 4(a+1)+ 5(a-1)+ 6(a-2) \\
&=  2{a-2 \choose 2} + 23a -6.
\end{align*}
Then
$$ W(G_2) - W(G_1) = 23a - 6 - 13a - 16 - 4a + 6 = 6a - 16$$
and finally for $a \geq 2$, we get
\begin{align*}
   (W(G_2) - W(G_1)) + &(|V(H_1)|-1)\cdot (t_{G_2}(v) - t_{G_1}(v)) \\&\geq  6a-16 + (|V(H_1)|-1)2a \ge 10a - 16 > 0. 
\end{align*}
This completes the proof.
\end{proof}

\begin{lemma} \label{lem:hodnota}
The Wiener index of the onion graph $\On(k,l,m)$ is equal to
$$ k^2 + 7k + 8 + \frac{l^3-l}{6} + m^2 + m \frac{l^2 + l - 2}{2}
+ (k+3)(\frac{l^2-l}{2} + ml) + (l+m-1)(3k+4).$$
\end{lemma}
\begin{proof}
We will compute the Wiener index of particular subgraphs and repeatedly use Lemma~\ref{lem:polansky}.

Let $u$ and $v$ be two antipodal vertices in $C_4$ and let $v'$ be the vertex of degree $k$ in $K_{1,k}$. We denote by $G_1$ a graph obtained from $K_{1,k}$ and $C_4$ by identifying $v' \in K_{1,k}$ with $v \in C_4$.
Note that $W(K_{1,k}) = k^2$. By Lemma~\ref{lem:polansky} we get
$$W(G_1) = k^2 + 8 + (k+1-1)4 + 3k = k^2 + 7k + 8.$$
Let $G_2$ be a broom with path of length $l$ and $m$ pendant vertices. We denote by $u_1$ and $u_l$ endpoints of $P_l$ in such way that pendant vertices are attached to $u_l$. Observe that $W(P_l) = {l + 1 \choose 3}$ and $W(K_{1,m}) = m^2$. For Wiener index of $G_2$ it holds by Lemma~\ref{lem:polansky} that
$$W(G_2) =  \frac{l^3-l}{6} + m^2 + (l-1)m + m \frac{l^2-l}{2}.$$
Note that we can get $G$ from $G_1$ and $G_2$ by identifying $u \in G_1$ with $u_1 \in G_2$. Note that
$|V(G_1)| = k+4$ and $|V(G_2)| = l+m$. For transmissions of $u$ and $u_1$, we have
\begin{align*}
 t_{G_1}(u) &= 3k+4, \\
 t_{G_2}(u_1) &= \frac{l^2-l}{2} + ml.
\end{align*}
Again by Lemma~\ref{lem:polansky} we get that Wiener index of $G$ is equal to
$$ k^2 + 7k + 8 + \frac{l^3-l}{6} + m^2 + m \frac{l^2 + l - 2}{2}
+ (k+3)(\frac{l^2-l}{2} + ml) + (l+m-1)(3k+4).$$
\end{proof}

The following lemma can be proved by an easy computation, so we skip the details.

\begin{lemma} \label{lem:transmise}
Let $G$ be an onion graph $\On(k,l,m)$. Let $v$ vertex on the cycle with $k$ pendant vertices and $u_l$ be the last vertex on the path of length $l$ with $m$ pendant vertices. Then
\begin{align*}
 t_{G}(v) &= k + 1 + {l+2\choose2} + m(l+2), \\
 t_{G}(u_l) &= m + {l+2\choose2} + l + k(l+2).
\end{align*}
\end{lemma}

\section{Proof of the main theorem}

We are now ready to prove Theorem~\ref{thm:main}.

\begin{proof}[of Theorem~\ref{thm:main}]
Let $G \in E_{p,q}$.
From the previous series of lemmata we get that $G$ is an onion graph with all pendant vertices are in part $Q$. Now we shall determine parameters of the onion for $G$.

Let $G$ be isomorphic to $\On(a,l,b)$ where $q-p = a+b$. Let $v$ be the  vertex on the cycle with $a$ pendant vertices and $u_l$ be the last vertex on the path of length $l$ with $b$ pendant vertices.
Since $q-p=a+b$, we get that $a = \lfloor\frac{q-p}{2}\rfloor$ implies $b = \lceil\frac{q+p}{2}\rceil$.
Further note that $\On(a,1,b) \cong \On(b,1,a)$ and thus for $l=1$, it might happen that $a = b+1$.

Suppose for a contradiction that $a \neq \lfloor \frac{q-p}{2} \rfloor$. By the above
reasoning there are only two cases to distinguish:
\begin{itemize}
    \item $a \ge b + 1$ and $l \ge 2$, or
    \item $b \ge a + 2$.
\end{itemize}

\emph{Case 1. } We assume that $a \geq b+1$ and $l \geq 2$.
We define two graphs $G_1 := \On(a-1,l,b)$ and $H_1:= K_2$.
Clearly, $G$ can be obtained from $G_1$ and $H_1$ by identifying $v \in G_1$ with one of the vertices of $H_1$.
Let us denote by $G'$ a graph obtained from $G_1$ and $H_1$ by identifying $u_l \in G_1$ with one of the vertices of $H_1$. We use Lemma~\ref{lem:transmise}  for parameters $k = a-1, l$ and $m= b$ and we get that
\begin{align*}
 t_{G_1}(v) &= a + {l+2\choose2} + b(l+2), \\
 t_{G_1}(u_l) &= b + {l+2\choose2} + l + (a-1)(l+2),\\
 t_{G_1}(u_l) - t_{G_1}(v) &= l(a-b) + (a-b) - 2 > 0.
\end{align*}
From Lemma~\ref{lem:du} we conclude that $W(G') > W(G)$, a contradiction.

\emph{Case 2. } We assume that $b \geq a+2$. Analogously to the previous case we define $G_2 := \On(a,l,b-1)$  and $H_2:= K_2$. Observe that $G$ can be obtained from $G_2$ and $H_2$ by identifying $u_l \in G_2$ with one of the vertices of $H_2$. We define $G''$ as a graph obtained from $G_2$ and $H_2$ by identifying $v \in G_2$ with one of the vertices of $H_2$. Again we use Lemma~\ref{lem:transmise}
for parameters $k = a, l$ and $m= b-1$ to compute the transmissions:
\begin{align*}
 t_{G_2}(v) &= a + 1 + {l+2\choose2} + (b-1)(l+2), \\
 t_{G_2}(u_l) &= b -1 + {l+2\choose2} + l + a(l+2),\\
 t_{G_2}(v) - t_{G_2}(u_l) &= l(b-a-2) + b-a> 0.
\end{align*}
By Lemma~\ref{lem:du} we get $W(G') > W(G)$, a contradiction.

Therefore, $a = \lfloor (q-p)/2 \rfloor$, $b=\lceil (q-p)/2 \rceil$ and
$$l = p+q-a-b-3 = p +q - (q-p) - 3 = 2p - 3.$$ From this we conclude that  $$ G \cong \On(\lfloor (q-p)/2 \rfloor, 2p-3, \lceil (q-p)/2
\rceil).$$ 
Using Lemma~\ref{lem:hodnota} we get the exact value of its
Wiener index which is equal to
\begin{align*}
&(2p-5)\cdot\lceil(q-p)/2\rceil\lfloor(q-p)/2\rfloor + (p-7)\lceil(q-p)/2\rceil + (13-7p)\lfloor(q-p)/2\rfloor + \\
&2p^2q + (q-p)^2 + 2p^3 - 37p + 66.
\end{align*}
This finishes the proof.
\end{proof}

\section{Conclusion and future work}

We obtained the extremal graphs and values for the maximum Wiener index on
unicyclic graphs with given bipartition. A natural next question is to
consider the class of cacti. A cactus graph is a graph where every edge
belongs to at most one cycle. Unicyclic graphs are precisely the cacti with
one cycle.

\begin{problem}
    What is the maximum Wiener index and the extremal graphs attaining such index
    for bipartite cacti graphs with given size of parts and number of cycles?
\end{problem}

\section*{Acknowledgments}

The first and the second author would like to acknowledge the support of the
grant SVV-2017-260452. The second author was
supported by Student Faculty Grant of Faculty of Mathematics and Physics,
Charles University.

\bibliographystyle{plain}
\bibliography{wiener_unicyclic}

\begin{thebibliography}{10}

\bibitem{bollobas2004extremal}
B.~Bollob{\'a}s.
\newblock {\em Extremal graph theory}.
\newblock Courier Corporation, 2004.

\bibitem{bonchev2002wiener}
D.~Bonchev.
\newblock The {W}iener number--some applications and new developments.
\newblock In {\em Topology in Chemistry}, pages 58--88. Elsevier, 2002.

\bibitem{dobrynin2001wiener}
A.~Dobrynin, R.~Entringer, and I.~Gutman.
\newblock Wiener index of trees: theory and applications.
\newblock {\em Acta Applicandae Mathematica}, 66(3):211--249, 2001.

\bibitem{dong2012maximum}
H.~Dong and B.~Zhou.
\newblock Maximum {W}iener index of unicyclic graphs with fixed maximum degree.
\newblock {\em Ars Combinatorica}, 103:407--416, 2012.

\bibitem{Du2012}
Z.~Du.
\newblock {Wiener indices of trees and monocyclic graphs with given
  bipartition}.
\newblock {\em International Journal of Quantum Chemistry}, 112:1598--1605,
  2012.

\bibitem{entringer1976distance}
R.~C. Entringer, D.~E. Jackson, and D.~A. Snyder.
\newblock Distance in graphs.
\newblock {\em Czechoslovak Mathematical Journal}, 26(2):283--296, 1976.

\bibitem{harary1959status}
F.~Harary.
\newblock Status and contrastatus.
\newblock {\em Sociometry}, 22(1):23--43, 1959.

\bibitem{hou2012maximum}
H.~Hou, B.~Liu, and Y.~Huang.
\newblock The maximum {W}iener polarity index of unicyclic graphs.
\newblock {\em Applied Mathematics and Computation}, 218(20):10149--10157,
  2012.

\bibitem{knor2014wiener}
M.~Knor and R.~{\v{S}}krekovski.
\newblock Wiener index of line graphs.
\newblock {\em Quantitative Graph Theory: Mathematical Foundations and
  Applications}, pages 279--301, 2014.

\bibitem{Knor2016}
M.~Knor, R.~{\v{S}}krekovski, and A.~Tepeh.
\newblock {Mathematical aspects of {W}iener index}.
\newblock {\em Ars Mathematica Contemporanea}, 11:327--352, 2016.

\bibitem{liu2011wiener}
M.~Liu and B.~Liu.
\newblock On the {W}iener polarity index.
\newblock {\em MATCH Commun. Math. Comput. Chem}, 66(1):293--304, 2011.

\bibitem{polansky1986}
O.~E. Polansky and D.~Bonchev.
\newblock The {W}iener number of graphs. i. general theory and changes due to
  some graph operations.
\newblock {\em MATCH Commun. Math. Comput. Chem}, 21(133-186):72, 1986.

\bibitem{vsoltes1991transmission}
L.~{\v{S}}olt{\'e}s.
\newblock Transmission in graphs: a bound and vertex removing.
\newblock {\em Mathematica Slovaca}, 41(1):11--16, 1991.

\bibitem{wiener}
H.~Wiener.
\newblock Structural determination of paraffin boiling points.
\newblock {\em Journal of the American Chemical Society}, 69(1):17--20, 1947.

\bibitem{Xu2014461}
K.~Xu, M.~Liu, K.~Ch.~Das, I.~Gutman, and B.~Furtula.
\newblock A survey on graphs extremal with respect to distance-based
  topological indices.
\newblock {\em MATCH Commun. Math. Comput. Chem}, 71(3):461--508, 2014.

\end{thebibliography}

\end{document}